\documentclass[12pt,leqno]{article}

\usepackage{graphicx, fullpage}
\usepackage{amsmath,amssymb,amsthm,amscd, bm}
\usepackage{fancyhdr}
\usepackage[mathscr]{eucal}
\usepackage{amsfonts}
\begin{document}
\parskip=6pt

\theoremstyle{plain}
\newtheorem{prop}{Proposition}
\newtheorem{lem}[prop]{Lemma}
\newtheorem{thm}[prop]{Theorem}
\newtheorem{cor}[prop]{Corollary}
\newtheorem{defn}[prop]{Definition}
\theoremstyle{definition}
\newtheorem{example}[prop]{Example}
\theoremstyle{remark}
\newtheorem{remark}[prop]{Remark}
\numberwithin{prop}{section}
\numberwithin{equation}{section}

\def\bC{\Bbb C}
\def\bK{\Bbb K}
\def\bR{\Bbb R}
\def\cB{\mathcal B}
\def\cH{\mathcal H}
\def\cQ{\mathcal Q}
\def\cK{\mathcal K}
\def\cL{\mathcal L}
\def\cM{\mathcal M}
\def\cS{\mathcal S}
\def\cU{\mathcal U}
\def\cW{\mathcal W}
\def\ep{\epsilon}
\def\os{\overline S}
\def\oS{\overline S}
\def\oU{\overline {\mathcal U}} 
\def\obU{\overline{\mathcal U}}
\def\ep{\epsilon}

\def\Rsa{R_{\text{sa}}}
\def\asa{A_{\text{sa}}}
\def\csa{C_{\text{sa}}}
\def\nul{\text{hull}}
\def\ker{\text{ker}}
\def\id{\text{id}}
\def\dim{\text{dim}}
\def\prim{\text{prim}}

\def\calg{C^*\text{--algebra}}
\def\calgs{C^*\text{--algebras}}
\def\salg{C^*\text{--subalgebra}}
\def\salgs{C^*\text{--subalgebras}}

\def\b1{\bold 1}
\def\tA{\widetilde A}
\def\tf{\widetilde f}
\def\Asa{\widetilde A_{\text{sa}}}
\def\Csa{\widetilde C_{\text{sa}}}
\def\im{\text{Im}}
\def\re{\text{Re}}
\def\m{\Bbb M_{n\text {sa}}}

\begin{titlepage}
\title{When is Every Quasi-multiplier a Multiplier?}
\author{Lawrence G.~Brown \thanks{AMS subject classication: 46L05.
}}
\thispagestyle{empty}

\end{titlepage}
\date{}
\maketitle
\center{Dedicated to the memory of Ronald G. Douglas}
\abstract
We answer the title question for $\sigma$--unital $\calgs$. The answer is that the algebra must be the direct sum of a dual $C^*$-algebra and a $C^*$-algebra satisfying a certain local unitality condition.  We also discuss similar problems in the context of Hilbert $C^*-$bimodules and imprimitivity bimodules and in the context of centralizers of Pedersen's ideal.
\endabstract

\section*{0 Introduction}
Let $A$ be a $\calg$ and $A^{**}$ its Banach space double dual, also known as its enveloping von Neumann algebra.  An element $T$ of $A^{**}$ is called a multiplier of $A$ if $Ta\in A$ and $aT\in A$, $\forall a\in A$. Also $T$ is a left multiplier if $Ta\in A$, $\forall a\in A$, $T$ is a right multiplier if $aT\in A$, $\forall a\in A$ and $T$ is a quasi-multiplier if $aTb\in A$, $\forall a, b\in A$. The sets of multipliers, left multipliers, right multipliers and quasi-multipliers are denoted respectively by $M(A), LM(A), RM(A)$, and $QM(A)$. More information about multipliers, etc. can be found in \cite[\S{3.12}]{P}.

We believe that quasi-multipliers were first introduced to operator algebraists  in \cite{AP}. It was shown there that a self-adjoint element $h$ of $A^{**}$ is a multiplier if and only if $\pm h$ satisfy a certain semicontinuity property, and self-adjoint quasi-multipliers are characterized similarly with a weaker semicontinuity property. The fact that these semicontinuity properties are in general different was one of the key ``complications'' discovered in \cite{AP}. We will not use semicontinuity theory in any proofs in this paper.

Multipliers of $\calgs$ have many important applications. In particular they play a crucial role in the theory of extensions of $\calgs$, as shown in \cite{Bu}, and they are used in $KK$-theory.  Quasi-multipliers, though less important, also have applications as shown, for example, in \cite{B1}, \cite{S}, and \cite{BMS}. (Note that \cite{BMS} contains the results of \cite{S}).

It is obvious that $QM(A)=M(A)$ if $A$ is commutative, and more generally if $A$ is $n$-homogeneous. Also $QM(A)=M(A)$ if $A$ is elementary, and therefore also if $A$ is dual; i.e., if $A$ is the direct sum of elementary $\calgs$. It follows from \cite[Theorem 4.9]{B1} that if $LM(A)=M(A)$ for  a $\sigma$-unital $\calg$ $A$, then also $QM(A)=M(A)$. (This is shown also without $\sigma$-unitality in Proposition 2.7 below.) Therefore it is sufficient to consider the title question.

My association with Ron Douglas was very beneficial to and influential in my career. In particular my interests in multipliers and Calkin algebras arose from this association.

\section{Preliminaries.}

A $\calg$ $A$ is called \textit{locally} \textit{unital} if there is a family $\{I_j\}$ of (closed, two-sided) ideals such that $(\sum I_j)^-=A$ and for each $j$ there is $u_j$ in $A$ such that $(\b1-u_j)I_j=I_j(\b1-u_j)=\{0\}$. Here $\b1$ is the identity of $A^{**}$.  Since this concept may not be completely intuitive, we will explore what it means.

\begin{prop}  
If $I$ is an ideal of $A$, then there is $u$ in $A$ such that $(\b1-u)I=I(\b1-u)=\{0\}$ if and only if there is an ideal $J$ such that $IJ=\{0\}$ and $A/J$ is unital.
\end{prop}   

\begin{proof}
If $u$ is as above, let $J$ be the closed span of $A(\b1-u)A$.  Then clearly $IJ=\{0\}$. Also the image of $u$ is an identity for $A/J$, since, for example, the fact that $((\b1-u)a)^*(\b1-u)a\in J$ implies that $(\b1-u)a\in J$. Conversely if $J$ is as above, let $u$ be an element of $A$ whose image is the identity of $A/J$, then $(\b1-u)I\subset I\cap J=\{0\}$, and similarly $I(\b1-u)=\{0\}$. 
\end{proof}

Note that it follows from the above that $u$ may be taken to be a positive contraction. The following lemma is undoubtedly known but we don't know a reference.

\begin{lem}  
If $I$ and $J$ are ideals of a $\calg$ $A$ such that $IJ=\{0\}$, $A/I$ is unital, and $A/J$ is unital, then $A$ is unital.
\end{lem}    

\begin{proof}
Let $u$ and $v$ be elements of $A$ such that the image of $u$ is the identity of $A/I$ and the image of $v$ is the identity of $A/J$. Then both $u$ and $v$ map to the identity of $A/(I+J)$. Therefore $u-v=x+y$ with $x\in I$ and $y\in J$. If  $w=u-x=v+y$, then $w$ gives the identity both modulo $I$ and modulo $J$. Therefore $w$ is an identity for $A$.
\end{proof}  

We denote by prim$\, A$ the primitive ideal space of $A$, the basic facts about prim$\, A$ can be found in \cite[\S 4.1]{P}.

\begin{prop}  
The $\calg$ $A$ is locally unital if and only if:
\begin{itemize}
\item[(i)] Every compact subset of prim$\, A$ has compact closure, and 
\item[(ii)] For every closed compact subset of prim$\, A$ the corresponding quotient algebra is unital.
\end{itemize}
\end{prop}   

\begin{proof}
If $A$ is locally unital let $\{I_j\}$ be as in the definition. Then $\{\prim I_j\}$ is an open cover of $\prim A$. If $K$ is a compact subset of prim$\, A$, then there are $I_{j_1},\dots,I_{j_n}$ such that $K\subset \bigcup^n_1\prim I_{j_l}$. By Proposition 1.1 and Lemma 1.2 there is an ideal $J$ such that $(I_{j_1}+\cdots+I_{j_n}) J=\{0\}$ and $A/J$ is unital. It follows that if $L=\text{hull} (J)$, then $L$ is compact and closed, and $L\supset \bigcup^n_1\prim I_{j_l}\supset K$. This implies both (i) and (ii). 

Now assume (i) and (ii). There is an open cover $\{U_j\}$ of $\prim A$, such that each $U_j$ is contained in a compact set $K_j$. If $I_j$ is the ideal corresponding to $U_j$, then $A=(\sum I_j)^-$.  If $L_j=\overline{K_j}$ and $J_j=\ker(L_j)$, then $A/J_j$ is unital and $J_jI_j=\{0\}$. 
\end{proof}

Note that a dual $\calg$ is locally unital if and only if all of the elementary $\calgs$ in its direct sum decomposition are finite dimensional.  This follows from the above proposition, or it can be deduced directly from the definition.

If $\mathcal{A}=\{A_x: x\in X\}$ is a continuous field of $\calgs$ over a locally compact Hausdorff space $X$, then the corresponding $\calg$ is the set of continuous sections of $A$ vanishing at $\infty$. Of course $n$-homogeneous $\calgs$ arise in this way, where each $A_x$ is isomorphic to the algebra of $n\times n$ matrices. The local unitality of such algebras is discussed in the next proposition.

\begin{prop}   
Let $A$ be the $\calg$ arising from a continuous field of $\calgs$ $\{A_x\}$ over a locally compact Hausdorff space $X$. If each $A_x$ is unital and if the identity section is continuous, then $A$ is locally unital. Conversely, if each $A_x$ is simple and $A$ is locally unital, then each $A_x$ is unital and the identity section is continuous. 
\end{prop}  

\begin{proof}
For the first statement let $\{U_j\}$ be an open cover of $X$ such that each $\overline{U_j}$ is compact. If $I_j$ is the set of continuous sections vanishing outside of $U_j$, then $I_j$ is an ideal of $A$ and $A=(\sum I_j)^-$. If $f_j$ is a continuous scalar--valued function on $X$ vanishing at $\infty$ such that $f_j(x)=1$, $\forall x\in U_j$, let $u_j=f_j\b1$. Then $u_j\in A$ and $(\b1-u_j)I_j=I_j(\b1-u_j)=\{0\}$. 

Now assume each $A_x$ is simple and $A$ is locally unital. Note that $\prim A$ can now be identified with $X$. If $K$ is a compact subset of $X$, then the corresponding quotient algebra is obtained from the restriction of the continuous field to $K$. It follows from Proposition 1.3 that this algebra is unital. Therefore each $A_x$, $x\in K$, is unital and the identity section is continuous on $K$. Since $X$ is locally compact, the result follows. 
\end{proof}

Note that $A$ arises from a continuous field of simple $\calgs$ over a locally compact Hausdorff space if and only if $\prim A$ is Hausdorff. We provide an example to show that the simplicity is necessary in the second statement of Proposition 1.4. Let $X=[0,1]$ and let $A_x=\bC\oplus\bC$ for $x\ne 0$. Let $A_0=\bC$, identified with $\bC\oplus\{0\}\subset \bC\oplus \bC$. So $A$ is the set of continuous functions $f$ from $[0,1]$ to $\bC\oplus\bC$ such that $f(0)\in \bC\oplus\{0\}$. Clearly the identity section of this continuous field is not continuous. But $A$ is commutative and therefore locally unital.

We now establish some notations and record some facts that will be used throughout the next section. Let $e$ be a strictly positive element of a $\sigma$-unital $\calg$ $A$. Strictly positive elements are discussed in \cite[\S 3.10]{P}. One property is that the kernel projection of $e$ in the von Neumann algebra $A^{**}$ is 0. Another is that the sets $eA, Ae$, and $eAe$ are dense in $A$. 

We will also use the concept of open projection, see \cite{A} or the end of \cite[\S 3.11]{P}. Certain projections in $A^{**}$ are called open, and there is an order--preserving bijection between open projectious and hereditary $\salgs$ of $A$. If $p$ is the open projection for the hereditary $\salg\, B$ then $B=(pA^{**}p)\cap A$ and any approximate identity for $B$ converses to $p$ in the strong topology of $A^{**}$. Also $p$ is central in $A^{**}$ if and only if $B$ is an ideal, and for general $p$ the central cover of $p$ in $A^{**}$ is the open projection for the ideal of $A$ generated by $B$. 

If $U$ is an open subset of $(0,\infty)$, then the spectral projection $\chi_U(e)$ is an open projection. The corresponding subalgebra $B$ is the hereditary $\calg\, B$ generated by $f(e)$, where $f$ is any continuous function on $[0,\infty)$ such that $U=\{x:f(x)\ne 0\}$. If $U=(\ep,\infty)$, we will denote the corresponding subalgebra by $B_\ep$, and if $U=(0,\ep)$ we will denote the subalgebra by $C_\ep$.  Also we denote by $I_\ep$ the ideal generated by $C_\ep$.  (The reason we are looking at subsets of $(0,\infty)$ instead of $[0,\infty)$ is that the kernel projection of $e$ is $0$.)

If $p$ and $q$ are open projections with corresponding subalgebras $C$ and $B$, let $X(p,q)$ denote the closed linear span of $CAB$. Two facts that we don't need are that $X(p,q)=(pA^{**}q)\cap A$ and that $a$ (in $A$) is in $X(p,q)$ if and only if $a^*a\in B$ and $aa^*\in C$. A fact that we do need is that $X(p,q)=\{0\}$ if and only if $p$ and $q$ are centrally disjoint in $A^{**}$. This follows from the fact that the strong closure of $X(p,q)$ in $A^{**}$ is $pA^{**}q$.

The following lemma is probably known, but we don't know a reference.

\begin{lem}   
If $A$ is an infinite dimensional $\calg$, then $A$ contains an infinite sequence $\{B_n\}$ of mutually orthogonal non-zero hereditary $\salgs$.
\end{lem}   

\begin{proof}
By \cite[4.6.14]{KR} $A$ contains a self-adjoint element $h$ whose spectrum, $\sigma(h)$, is infinite. Therefore $\sigma(h)$ contains a cluster point $x_0$. Then there is a sequence $\{x_n\}\subset\sigma(h)$ such that $x_n\ne x_0$, $x_n\ne x_m$ for $n\ne m$, $x_n\ne 0$, and $\{x_n\}$ converges to $x_0$. It is a routine exercise to find mutually disjoint open sets $U_n$ such taht $x_n\in  U_n$. For each $n$ find a continuous function $f_n$ such that $f_n(x_n)\ne 0$, $f_n(0)=0$, and $f_n(x)=0$ for $x$ not in $U_n$. Then let $B_n$ be the hereditary $\calg$ generated by $f_n(h)$.
\end{proof}

\section{Results and Concluding Remarks.}  

Throughout this section, up to and including the proof of theorem 2.4, $A$ is a $\sigma$--unital $\calg$, $e$ is a strictly positive element of $A$, and the notations of the previous section apply. The following is the main lemma.

\begin{lem}   
If $\ep>0$ and $\{B_n\}$ is an infinite sequence of mutually orthogonal hereditary $\calgs$ of $B_\ep$ such that $B_n\cap I_{\frac 1n}\ne\{0\}$, $\forall n$, then $QM(A)\ne M(A)$. 
\end{lem}  

\begin{proof}
Let $p_n=\chi_{(0,\frac 1n})(e)$ and let $q_n$ be the open projection for $B_n$. Since $B_n\cap I_{\frac 1n}\ne\{0\}$ and the central cover of $p_n$ is the open projection for $I_{\frac 1n}$, $p_n$ and $q_n$ are not centrally disjoint in $A^{**}$. Therefore $X(p_n, q_n)\ne \{0\}$. Since $\chi_{(\theta,\frac 1n)}(e)$ converses to $p_n$ as $\theta\searrow 0$, it follows that also $X(\chi_{(\theta,\frac 1n)}(e), q_n)\ne 0$ for $\theta$ sufficiently small. Then we can recursively choose $n_k$ and $\theta_k$ such that $n_k\to\infty$, $\sum_1^\infty$ $\frac {1}{n_k}<\infty$, $0<\theta_k<\frac {1}{n_k}$, $\frac {1}{n_{k+1}}<\theta_k$, and $X(\chi_{(\theta_k, \frac{1}{n_k}})(e), q_{n_k})\ne \{0\}$. Choose $a_k\in X(\chi_{(\theta_k, \frac{1}{n_k}})(e), q_{n_k})$ such that $\|a_k\|=1$. Then the $a_k$'s are mutually orthogonal in the sense that $a_k^*a_l=a_ka^*_l=0$ for $k\ne l$. Thus $T=\sum^\infty_1 a_k$ exists in $A^{**}$.  Since $eT=\sum^\infty_1 e a_k$ and $\|e a_k\|\le \frac{1}{n_k}$, then $e T\in A$. Since $Ae$ is dense in $A$, this implies that $T\in RM(A)\subset QM(A)$. We claim that $T\not\in M(A)$. If $f$ is a continuous function such that $f(0)=0$ and $f(x)=1$ for $x\ge \ep$, then $T f(e)=T$. Thus $T\in M(A)$ implies $T\in A$. 
To see that this is not so, choose continuous functions $f_k$ such that $f_k(0)=0$, $f_k(x)=1$ for $\theta_k\le x\le\frac{1}{n_k}$, $f_k(x)=0$ for $x\ge \frac {2}{n_k}$, and $\|f_k\|=1$. Then $f_k(e)e\to 0$ and hence $f_k(e)b\to 0$ , $\forall b\in A$ (since $eA$ is dense in $A$). But $\|f_k(e)T\|\ge \|f_k(e)a_k\|=\|a_k\|=1$. 
\end{proof}

\begin{lem}   
If $QM(A)=M(A)$ and $I=\bigcap^\infty_1 I_{\frac 1n}$, then $I$ is a dual $\calg$ and a direct summand of $A$.
\end{lem}   

\begin{proof}
If $\ep>0$, then it is impossible to find infinitely many non-zero mutually orthogonal hereditary $C^*$--subalgebras of $B_\ep \cap I$. By Lemma 1.5 $B_\ep\cap I$ is finite dimensional. This implies that id$\, (B_\ep\cap I)$, the ideal generated by $B_\ep\cap I$, is the direct sum of finitely many elementary $\calgs$. Since $I$ is the limit of id$\, (B_\ep\cap I)$ as $\ep\searrow 0$, it follows that $I$ is dual. If $z$ is the open central projection for $I$, the fact that $B_\ep\cap I$ is finite dimensional implies that $z\chi_{(\ep,\infty)}(e)$ is a finite rank projection in the dual $\calg\, I$. It then follows from spectral theory that $ze\in I$. This implies that $z\in M(A)$, whence $I$ is a direct summand of $A$. 
\end{proof}

\begin{lem}   
If $QM(A)=M(A)$ and $\bigcap^\infty_1I_{\frac 1n}=\{0\}$, then $\forall \ep>0, \exists \delta>0$ such that $I_\delta B_\ep=\{0\}$. 
\end{lem}   

\begin{proof}
If $I_\delta\cap B_\ep$ is finite dimensional for some $\delta>0$, then dim$\,(I_{\frac 1n}\cap B_\ep)$ must stabilize at some finite value as $n$ increases. It follows that the set $I_{\frac 1n}\cap B_\ep$ also stabilizes, and hence $I_{\frac 1n}\cap B_\ep=\{0\}$ for $n$ sufficiently large. Therefore we may assume $I_\delta\cap B_\ep$ is infinite dimensional, $\forall \delta>0$. 

\underbar{Case 1.}\ There is $\delta>0$ such that $B_\ep\cap I_\delta\cap I_{\frac 1n}$, is an essential ideal of $B_\ep\cap I_\delta$, $\forall n$. Then, using Lemma 1.5, choose a sequence $\{B_n\}$ of non-zero mutually orthogonal hereditary $C^*$--subalgebras of $B_\ep\cap I_\delta$. By the essential property $B_n\cap I_{\frac 1n}\ne 0$, $\forall n$. So Lemma 2.1 gives a contradiction. 

\underbar{Case 2.}\ For each $\delta>0$, there is $n$ such that $B_\ep\cap I_\delta\cap I_{\frac 1n}$ is not essential in $B_\ep\cap I_\delta$. Then we can construct $n_k$ recursively so that $n_{k+1}>n_k$ and $B_\ep\cap I_{\frac {1}{n_{k+1}}}$ is not essential in $B_\ep\cap I_{\frac {1}{n_k}}$. Then for each $k$ there is a non-zero hereditary $C^*$--subalgebra $B_k$ of $B_\ep\cap I_{\frac {1}{n_k}}$ such that $B_k$ is orthogonal to $I_{\frac {1}{n_{k+1}}}$. Again Lemma 2.1 produces a contradiction.
\end{proof}

\begin{thm}  
If $A$ is a $\sigma$-unital $\calg$, then $QM(A)=M(A)$  if and only if $A$ is the direct sum of a dual $\calg$ and a localle unital $\calg$. 
\end{thm}   

\begin{proof}
Assume $QM(A)=M(A)$. By Lemma 2.2 $A=I\oplus A_1$ where  $I$ is a dual $\calg$ and $A_1$ satisfies the hypothesis of Lemma 2.3. For  each $\delta>0$, $A_1/ I_\delta$ is unital, since $f(e)$ maps to a unit for $A_1/I_\delta$ for any continuous function $f$ such that $f(x)=1$ for $x\ge\delta$. So if $J_\ep$ is the ideal generated  by $B_\ep$ and $I_\delta B_\ep=\{0\}$, then by Lemma 1.1, there is $u$ in $A$ such that $(\b1-u)J_\ep=J_\ep(\b1-u)=\{0\}$. Since $A_1=(\sum J_{\frac 1n})^-$, we have shown that $A_1$ is locally unital. 

Now assume $A=A_0\oplus A_1$ where $A_0$  is dual and $A_1$ is locally unital. (For this part we don't need $\sigma$--unitality.) Then $QM(A_0)=M(A_0)$, since $A_0$ is an ideal in $A_0^{**}$. Let $A_1=(\sum J_\alpha)^-$, as in the definition of locally unital. For each $\alpha$, the weak closure of $I_\alpha$ in $A_1^{**}$ will be denoted by $I_\alpha^{**}$ (to which it is isomorphic). Then $I_\alpha^{**}$ is an ideal in $A_1^{**}$. If $(u_\alpha-\b1)I_\alpha=\{0\}$, then also $(u_\alpha-\b1)I_\alpha^{**}=\{0\}$. If $T\in QM(A_1)$ and $x\in I_\alpha$, then since $Tx\in I_\alpha^{**}$, $Tx=u_\alpha Tx\in A_1$. Therefore $T\in LM(A)$. A symmetrical proof shows also that $T\in RM(A)$. 
\end{proof}

\begin{cor}   
If $A$ is a $\sigma$-unital simple $\calg$, then $QM(A)=M(A)$ if and only if $A$ is either elementary or unital.
\end{cor}   

Combining the theorem with \cite[Theorem 3.27]{B2}, we obtain:

\begin{cor}   
If $A$ is a $\sigma$-unital $\calg$, then the middle and weak forms of semicontinuity coincide in $A^{**}$ if and only if $A$ is the direct sum of a dual $\calg$ and a locally unital $\calg$.
\end{cor}   

It occurred belatedly to us that since many of the applications of quasi-multipliers concern quasi-multipliers of imprimitivity bimodules or Hilbert $C^*$--bimodules, it might make sense to consider the title question in a broader context. In proposition 2.7 below, whose proof is purely formal, the imprimitivity bimodule case is reduced to what has already been done. Example 2.8 below deals with what are probably the the simplest examples of Hilbert $C^*$--bimodules that are not impimitivity bimodules.  Although we have a solution for these examples, it has not immediately inspired a conjecture for the general case.  But Example 2.8 does suggest that it may be easier to find when $QM(X)=M(X)$ than to answer the constituent questions whether $LM(X)\subset RM(X)$ or $RM(X)\subset LM(X)$.

Hilbert $C^*$--bimodules were introduced in \cite{S} as a generalization of imprimitivity bimodules. If $X$ is an $A-B$ Hilbert $C^*$--bimodule, then $X$ has a linking algebra $L$. Then $L$ is a $\calg$ endowed with two multiplier projections $p$ and $q$, such that $pLp$ is identified with $A$, $qLq$ is identified with $B$, and $pLq$ is identified with $X$. The existence of $L$ may be taken as a working definition of Hilbert $C^*$--bimodule. Then $X$ is an $A-B$ imprimitivity bimodule if and only if each of $A$ and $B$ generates $L$ as an ideal. Linking algebras of imprimitivity bimodules were introduced in \cite{BGR}.

If $X$ and $L$ are as above, then we define $M(X)=pM(L)q$, $LM(X)=pLM(L)q$, $RM(X)=pRM(L)q$, and $QM(X)=pQM(L)q$.
 Note that $X^{**}$ can be identified with $pL^{**}q$. It is not hard to see that for $T\in X^{**}$, $T\in M(X)$ if and only if $aT\in X$ and $Tb\in X$, $\forall a\in A, b\in B, T\in LM(X)$ if and only if $Tb\in X$, $\forall b\in B$, $T\in RM(X)$ if and only if $aT\in X$, $\forall a\in A$, and $T\in QM(X)$ if and only if $aTb\in X$, $\forall a\in A$, $b\in B$.
Because it is no longer true that $RM(X)=LM(X)^*$, there are more than one question to consider. Since $M(X)=LM(X)\cap RM(X)$, there are actually only two questions. Namely, we ask where there $QM(X)=LM(X)$, which turns out to be equivalent to $RM(X)\subset LM(X)$, and whether $QM(X)=RM(X)$, which is equivalent to $LM(X)\subset RM(X)$. 

\begin{prop}  
Let $A$ be a $\calg$.
\begin{itemize}
\item[(i)]\ Then $LM(A)=M(A)$ if and only if $QM(A)=M(A)$. 

Let $X$ be an $A-B$ Hilbert $C^*$--bimodule.

\item[(ii)]\ Then $QM(X)=LM(X)$ if and only if $RM(X)\subset LM(X)$.

\item[(iii)]\ Then $QM(X)=RM(X)$ if and only if $LM(X)\subset RM(X)$

Let $X$ be an $A-B$ imprimitivity bimodule and $L$ its linking algebra. 

\item[(iv)]\ Then $QM(X)=LM(X)$ if and only if $QM(A)=M(A)$. 

\item[(v)]\ Then $QM(X)=RM(X)$ if and only if $QM(B)=M(B)$.
  
\item[(vi)]\ Then $QM(X)=M(X)$ if and only if $QM(L)=M(L)$. 
\end{itemize}
\end{prop}   

\begin{proof}
\begin{itemize}
\item[(i)]\ Assume $LM(A)=M(A)$ and $T\in QM(A)$. If $a\in A$, then $aT\in LM(A)$. Hence $aT\in M(A)$ and $AaT\subset A$. Since $A^2=A$, this implies $QM(A)\subset RM(A)$. Since $QM(A)=QM(A)^*$ and $LM(A)=RM(A)^*$, we also have $QM(A)\subset LM(A)$, whence $QM(A)=M(A)$.

\item[(ii)]\ Since $RM(X)\subset QM(X)$, one direction is obvious. So assume $RM(X)\subset LM(X)$ and let $T\in QM(X)$. For $b\in B$, $Tb\in RM(X)$. Therefore $TbB\subset X$. Since $B^2=B$, it follows that $T\in LM(X)$.

\item[(iii)]\ is similar to (ii).

\item[(iv)]\ First assume $QM(X)=LM(X)$ and let $T\in QM(A)$. If $x\in X$, then $Tx\in RM(X)\subset QM(X)$. Therefore $Tx\in LM(X)\subset LM(L)$. So if $y\in X$, then $Txy^*=T\langle x,y\rangle_A\in A$. Since $X$ is an imprimitivity bimodule, $\langle X,X\rangle_A$ spans a dense subset of $A$. Thus we have shown  $QM(A)=M(A)$.

Now assume $QM(A)=M(A)$ and $T\in QM(X)$. If $x\in X$, then $Tx^*\in RM(A)\subset QM(A)$. Therefore $Tx^*\in M(A)\subset M(L)$. So for $y\in X$, $Tx^*y=T\langle x,y\rangle_B\in X$. Since the span of $\langle X,X\rangle_B$ is dense in $B$, this shows that $T\in LM(X)$. 

\item[(v)]\ is similar to (iv)

\item[(vi)]\ Since one direction is obvious, assume $QM(X)=M(X)$, by (iv) and (v), then $QM(A)=M(A)$ and $QM(B)=M(B)$. Note that $X^*$ is a $B-A$ Hilbert $C^*$--bimodule, $QM(X^*)=QM(X)^*$, and $M(X^*)=M(X)^*$. So for each of the four components of $L$, we have that every quasi-multiplier is a multiplier. Therefore $QM(L)=M(L)$.
\end{itemize}
\end{proof}

In connection with (iv) and (v) note that $A$ is strongly morita equivalent to $B$ if and only if an $A-B$ imprimitivity bimolule exists. Several important properties of $\calgs$ are preserved by strong morita equivalence. The property of being a dual $\calg$ is so preserved but local unitality is not preserved. It is easy to construct examples of strongly morita equivalent separable $\calgs$ $A$ and $B$ such that $QM(A)=M(A)$ but $QM(B)\ne M(B)$

\begin{example}

Let $H_1$ and $H_2$ be Hilbert spaces and let $X=\bK(H_2,H_1)$, the space of compact operators from $H_2$ to $H_1$. Then $X$ is a $\bK(H_1)-\bK(H_2)$ imprimitivity bimodule, but $X$ can also be made into an $A-B$ Hilbert $C^*$--bimodule in many ways. We just take $A$  and $B$ to be $C^*$--subalgebras of $B(H_1)$ and $B(H_2)$ such that $A\supset\bK(H_1)$ and $B\supset \bK(H_2)$. Now $X^{**}$ can be identified with $B(H_2, H_1)$, the space of all bounded linear operators from $H_2$ to $H_1$. (Note that $L\subset B(H_1\oplus H_2)$.) For $T\in X^{**}$, $T\in LM(X)$ if and only if $Tb$ is compact, $\forall b\in B$. This is equivalent to $T^*TB\subset \bK(H_2)$. In other words, $T\in LM(X)$ if and only if the image of $T^*T$ in the Calkin algebra of $H_2$ is orthogonal to the image of $B$ in this Calkin algebra. Similarly $T\in RM(X)$ if and only if the image of $TT^*$ in the Calkin algebra of $H_1$ is orthogonal to the image of $A$ in this Calkin algebra.   
To analyze this, we add some reasonable hypotheses.  Assume that each of $H_1$ and $H_2$ is separable and infinite dimensional and that each of $A$ and $B$ is $\sigma-$unital.  Note that any non-zero projection in the Calkin algebra can be lifted to a projection which is necessarily of infinite rank and that if $P$ and $Q$ are infinite rank projections in $B(H_1)$ and $B(H_2)$ there is a partial isomety $U$ such that $UU^*=P$ and $U^*U=Q$.  Looking at $H_1$ for example, we see that there is a Calkin projection, namely $\b1$, which fails to annihilate the Calkin image of $A$ if and only if $A\ne \bK(H_1)$.  Also it was essentially shown in \cite{BDF} that there is a non-zero Calkin projection which does annihilate the Calkin image of $A$ if and only if $\b1\notin A$.  Then we see that $LM(X)\subset RM(X)$ if and only if either $\b1\in B$ (which causes $LM(X)$ to be small) or $A=\bK(H_1)$ (which causes $RM(X)$ to be big), and $RM(X)\subset LM(X)$ if and only if either $\b1\in A$ or $B=\bK(H_2)$.  Also $QM(X)=M(X)$ if and only if either both $A$ and $B$ contain $\b1$ or both $A=\bK(H_1)$ and $B=\bK(H_2)$.  Of course the last case is the case when $X$ is an imprimitivity bimodule.
\end{example}

Although Example 2.8 may deal with the simplest examples, it may actually be exceptional.  In support of this, we point out that in \cite{B3} the case where $A$ has an infinite dimensional elementary direct summand was the ``bad'' case. 

As an after--afterthought, there is another way to generalize the problem of this paper. Namely, consider questions like the title question in the context of centralizers of Pedersen's ideal.  We will discuss this informally.  The interested reader can fill in the details with the help of the discussion of centralizers in \cite[\S 3.12]{P} and the discussion of Pedersen's ideal in \cite[\S 5.6]{P}.  We think the appropriate question is whether every locally bounded quasi-centralizer of Pedersen's ideal comes from a double centralizer, and the answer, if $A$ is $\sigma-$unital, is the same as before.  Namely, $A$ must be the direct sum of a dual $\calg$ and a locally unital $\calg$.  If we drop the local boundedness hypothesis, then $A$ must be locally unital.  The reason for this last assertion is that it was shown in \cite{LT} that every double centralizer of Pedersen's ideal is locally bounded, and it was pointed out in\cite{B3} that if $A$ has an infinite dimensional elementary direct summand, then there definitely are non-locally bounded quasi-centralizers of Pedersen's ideal.  Our reason for preferring the first version of the question is that we consider non-locally bounded centralizers to be pathological.

To prove the forward direction of the assertion above, note that the hypothesis implies that every $T$ in $QM(A)$ comes from a double centralizer $C$.  Since $C$ is locally bounded and agrees with the bounded $T$ on Pedersen's ideal, then $C$ is bounded.  Thus $C$ comes from a multiplier and $QM(A)=M(A)$.  For the converse, the case of dual $\calgs$ is trivial.  If $A$ is locally unital, let $I_j$ and $u_j$ be as in the definition of local unitality.  It is not hard to see that each $I_j$ is contained in Pedersen's ideal and that $u_j$ may be taken in Pedersen's ideal.  (Actually Pedersen's ideal is $\sum I_j$ in this case.)  Then the argument for the converse direction of Theorem 2.4 applies.

\end{document}